\documentclass[a4paper,oneside,leqno,11pt]{article}
\usepackage[utf8]{inputenc}
\usepackage{fancyhdr}
\usepackage{amsmath}
\usepackage{amsfonts, amssymb}
\usepackage{amsthm}
\usepackage{mathrsfs}
\usepackage{mathtools}
\usepackage{graphics}
\usepackage{amssymb}
\usepackage{relsize}
\usepackage[all]{xy}
\usepackage{color}
\usepackage[small, bf, margin=90pt, tableposition=bottom]{caption}
\usepackage{microtype}
\usepackage{float}   
\usepackage{verbatim}
\usepackage{hyperref}
\graphicspath{{imagenes/}}
\usepackage{subcaption}
\usepackage{float}

\newcommand\N{\mathcal{N}}
\newcommand\ZZ{\mathbb{Z}}

\newcommand\RR{\mathbb{R}}
\newcommand\CC{\mathbb{C}}
\newcommand\PP{\mathbb{P}}

\newcommand\C{\mathcal{C}}
\newcommand\DD{\mathcal{D}}

\newcommand\OO{\mathcal{O}}
\newcommand\HH{\mathcal{H}}
\newcommand\GG{\mathcal{G}}
\newcommand\ov{\overline}
\newcommand\wt{\widetilde}
\newcommand\FF{\mathcal{F}}
\newcommand\LL{\mathcal{L}}
\newcommand\TT{\mathcal{T}}
\newcommand\UU{\mathcal{U}}
\newcommand\X{\mathfrak{X}}
\newcommand\ii{\imath}

\DeclareMathOperator{\codim}{codim}

\DeclareMathOperator{\eff} {Eff}
\DeclareMathOperator{\Hom}{Hom}

\theoremstyle{plain} 

\newcounter{Theorem}

\newtheorem{nteo}{Theorem}
\newtheorem*{ncorollary}{Corollary}

\newtheorem*{nobs}{Remark}

\newtheorem{teo}[Theorem]{Theorem} 
\newtheorem{prop}[Theorem]{Proposition}

\newtheorem{lema}[Theorem]{Lemma}
\newtheorem{corollary}[Theorem]{Corollary}

\theoremstyle{definition}
\newtheorem{deff}[Theorem]{Definition}
\newtheorem{obs}[Theorem]{Remark}
\newtheorem{ej}[Theorem]{Example}

\newenvironment{dem}
	{\noindent\textit{Proof.}}{\qed \vspace{.5cm}}

\title{Toric foliations with split tangent sheaf}
\author{Sebastián Velazquez $^1$}
\date{ }

\begin{document}

\maketitle

\begin{abstract}
 We study holomorphic foliations of arbitrary codimension in smooth complete toric varieties. We show that split foliations are stable if some good behaviour of their singular set is provided. As an application of these results, we exhibit irreducible components of the space of foliations that arise as pullbacks of some special $T$-invariant subvarieties.
\end{abstract}

\section{Introduction}
The problem of understanding the geometry of moduli spaces of foliations in projective spaces first appeared in the work of J. P. Jouanolou. Ever since, this problem has remained a very active area of research. For more general classes of varieties, however, this topic has received considerably less attention.
This article deals with singular foliations in smooth complete toric varieties. More specifically, we will address the problem of understanding the geometry of the moduli space of codimension $q$ singular foliations $\FF_q(X,\LL)$ by studying the family of foliations whose tangent sheaves split as a sum of line bundles.

Toric varieties have proven to be a rich class of examples in algebraic geometry. This is mainly due to the fact that its geometry is encoded in a combinatorial object, namely its \textit{fan}. Moreover, the appropriate use of Cox coordinates makes calculations on these varieties even more feasible. 

This work was mainly motivated by \cite{CP}. There, the authors show that the set of foliations with split tangent sheaf has non trivial interior in $\FF_q(\PP^n,d)$ by means of giving sufficient conditions for a foliation to belong to this particular set. We will use the notation $S(\DD)$ and $K(\DD)$ for the singular set and the Kupka set of a distribution $\DD$ on $X$ respectively (see Section 2.3 for a precise definition on these objects).
In Section 3.1 we will use the same ideas in order to prove the natural generalizations of these results to the toric setting, namely:

\begin{nteo}\label{stab1intro} Let $X$ be a complete toric variety of dimension $n\geq 3$ and $\alpha_1,\dots\alpha_{n-1}\in Pic(X)$ such that $h^1(X,\OO_X(-\alpha_i))=0$. Then for every $\FF\in \FF_1(X, \sum \alpha_i - \omega_X)$ 
satisfying $codim(S(\FF)\setminus \ov{K(\FF)})\geq 3$ and 
\[ \mathcal{T}\FF \simeq \bigoplus_{i=1}^{n-1}\OO_{X}(\alpha_i)\]
there exists a Zariski open set $\UU\subseteq \FF_1(X, \sum \alpha_i - \omega_X)$ containing $\FF$ such that
$\TT\FF'\simeq \TT\FF$ for every $\FF' \in \UU$.
\end{nteo}

\begin{nteo} \label{stab2intro} Let $q\geq 2$ be an integer, $X$ a complete toric variety of dimension $n\geq 3$ and $\alpha_1,\dots\alpha_{n-q}\in Pic(X)$ such that $h^1(X,\OO_X(-\alpha_i))=0$. Then for every distribution $\DD\in \DD_q(X, \sum \alpha_i - \omega_X)$ 
satisfying $\codim(S(\DD))\geq 3$ and 
\[ \mathcal{T}\DD \simeq \bigoplus_{i=1}^{n-q}\OO_{X}(\alpha_i)\]
there exists a Zariski open set $\UU\subseteq \DD_q(X, \sum \alpha_i - \omega_X)$ containing $\DD$ such that
$\TT\DD'\simeq \TT\DD$ for every $\DD' \in \UU$.
\end{nteo}

\begin{nobs} The hypothesis on the vanishing of the cohomology groups in the theorems above can be replaced by some weaker hypothesis, see Remark \ref{obs} at the end of Section 3.1.
\end{nobs}

Recall that every $n$-dimensional toric variety $X$ admits a finite number of $T$-invariant irreducible hypersurfaces $D_1,\dots,D_{n+s}$ (indexed by the 1-dimensional cones in its fan).  We will say that $D_j$ is \textit{maximal} if it satisfies $\dim H^0(X,\OO(D_i-D_j)) =0$
for every $D_j\not\sim D_i$. We will see that such elements always exist. If $\{ Di\}_{i\in S}$ is a collection of $T$-divisors with non-trivial intersection, then $D_S:=\cap_{i\in S} D_i$ is again a smooth complete toric variety.
As an application of the theorems above, we will show that the set of foliations which arise as linear pullbacks of foliations in some of these intersections fills out irreducible components of the corresponding moduli space of foliations in $X$. More precisely, in Section 3.2 we will prove:

\begin{ncorollary} Let $X$ be a complete toric variety of dimension $n\geq 3$ and $\{D_i\}_{i\in S}$ a set of maximal elements such that
$\dim (D_S) \geq 2$. Let $\beta\in Pic(D_S)$ and $\mathcal{C}\subseteq \FF_q(D_S, \beta )$ an irreducible component. Let $\alpha\in Pic(X)$ be the pullback of $\beta$ by a generic linear projection. If the generic element of $\mathcal{C}$ satisfies the hypotheses of Theorem \ref{stab1intro} (for $q=1$) or Theorem \ref{stab2intro} (for $q>1$)
then there exists an irreducible component of $\FF_q(X,\alpha)$ such that its generic element is a linear pullback of an element of $\C$. 
\end{ncorollary}

\vspace{1cm}

\textbf{Acknowledgements.} The author is very grateful to Fernando Cukierman, Jorge Vit\'orio Pereira, Federico Quallbrunn, César Massri and the anonymous referee for their useful comments and suggestions.

\section{Preliminaries}

\subsection{Toric varieties and Euler sequences}

In this section we discuss all the facts concerning toric varieties that will be used 
afterwards. For more details see \cite{CLS}. We will follow the notation used
in \cite{CLS}. 

Let $X=X_\Sigma$ be the toric variety associated to a fan $\Sigma$ in $\RR^n$ and $M:=Hom_\ZZ(T,\CC^*)$ the character lattice of its torus $T$. We will assume that $X$ is smooth and complete 
(or equivalently, that $\Sigma$ covers $\RR^n$ and the set of rays in every cone can be extended to a basis of $\ZZ^n$). Recall that the subgroup $Div_T(X)$ of divisors 
that are fixed by the torus action is freely generated by the elements $D_i$ associated to the rays of $\Sigma$, i.e., it is isomorphic to $\ZZ^{n+s}$, where $n+s$ is the number of rays 
in $\Sigma$.

The morphism $M\to Div_T(X)$ sending $m\mapsto div(\chi^m)$ together with the restriction of the quotient map $Div_T(X)\to Pic(X)$ fit together in the exact sequence
\[ 0\longrightarrow M\longrightarrow Div_T(X) \longrightarrow Pic(X)\longrightarrow 0. \]
This is the very basis of the construction of homogeneous coordinates: if we denote $G=Hom(Pic(X), \CC^*)$ then applying the functor $Hom(-,\CC^*)$  we get 
\[ 1\longrightarrow G \longrightarrow (\CC^*)^{n+s} \longrightarrow T \longrightarrow 1,\]
which is also exact. With this in mind, we can think of $T$ as the quotient of $(\CC^*)^{n+s}$ by the subgroup $G$. The construction of homogeneous coordinates in the sense of \cite{Cox} is just an extension of this presentation, i.e., a good geometric quotient $\pi: \CC^{n+s}\setminus Z \to X$ such that the diagram
\[
\xymatrix{
1 \ar[r]& G\ar[r]\ar@{-}[d]& (\CC^*)^{n+s}\ar[r]\ar@{^{(}->}[d] & T \ar[r]\ar@{^{(}->}[d] & 1   \\
& G \ar[r] & \CC^{n+s}\setminus Z   \ar[r]_\pi  & X_\Sigma \\
}
\]
commutes. 
\begin{obs} \label{fijariso} Since $X$ is smooth and has a point that is fixed by the torus action (or equivalently, $\Sigma$ has a cone of dimension $n$) we can assure that $Pic(X)$ is free. \textit{We will now fix an isomorphism} $Pic(X)\simeq \ZZ^s$. This choice also induces a canonical isomorphism $G=Hom(Pic(X),\ZZ)\simeq (\CC^*)^s$.
\end{obs}
The details of the construction of such quotient will not be explained here. The reader that is not familiarized with these ideas is referred to [Chapter 5, \cite{CLS}]. However, one cannot fail to mention that $Z$ is just a union of linear coordinate subspaces  and satisfies $\codim(Z)\geq 2$. This last fact tells us that the coordinate ring of $\CC^{n+s}\setminus Z$ is the polynomial ring $\CC[x_1,\dots,x_{n+s}]$. 

Being $Pic(X)$ a free abelian group, the natural evaluation map $ev:Pic(X)\to Hom(G,\CC^*)$ from the Picard group of $X$ into the character lattice of $G$ is an isomorphism. We will use the notation $\chi^\alpha$ to denote the character induced by $\alpha\in Pic(X)$. If we look closely at the diagrams above we can deduce that the action of $G$ on $\CC^{n+s}\setminus Z$ is given by 
\[ g\cdot (p_1,\dots, p_{n+s}) = (\chi^{[D_1]}(g)p_1,\dots, \chi^{[D_{n+s}]}(g)p_{n+s}),\]
where $g\in G$ and $[D_i]\in Pic(X)$ is the class of the $T$-invariant divisor $D_i$. 
Having fixed the isomorphism in Remark \ref{fijariso}, we can describe this action more concretely:  we can replace the embedding $G\longrightarrow (\CC^*)^{n+s}$ by a group homomorphism $(\CC^*)^s\longrightarrow (\CC^*)^{n+s}$, so we can think of the image of this morphism acting on $\CC^{n+s}\setminus Z$ by diagonal matrices:
\[ (t_1,\dots, t_s)\cdot (p_1,\dots, p_{n+s}) = (t_1^{a^1_1}\dots t_s^{a^s_1}p_1,\dots, t_1^{a^1_{n+s}}\dots t_s^{a^s_{n+s}}p_{n+s}),\]
where $(a_i^1,\dots,a_i^s)=[D_i]\in Pic(X)$.

At the level of coordinate rings, this action can be simultaneously diagonalized, i.e., we get a decomposition 
\[ \CC[x_1,\dots,x_{n+s}]=\bigoplus_{\alpha \in Pic(X)} S_\alpha ,\]
where $S_\alpha=\{ f \in \CC[x_1,\dots.x_{n+s}] | f(g\cdot x)=\chi^{\alpha}(g)f(x) \}$. The ring $\CC[x_1,\dots,x_{n+s}]$ equipped with this grading is the \textit{Cox ring} of $X$. A good feature of this grading is that there are natural isomorphisms
\[ H^0(X,\LL)\simeq S_\LL .\] 
In particular, we have $S_0=H^0(X,\OO_X)=\CC$ since $X$ is complete.
\begin{obs}It is actually easy to calculate the degree of an element since the coordinate functions satisfy $\deg(x_i)=[D_i]$.
\end{obs}
 
With respect to the theory of foliations, the main advantage of having homogeneous coordinates is the \textit{generalized Euler sequence}
\begin{equation} \label{Om}0 \longrightarrow \Omega^1_{X} \longrightarrow \bigoplus_{i=1}^{n+s} \OO_{X}(-D_i) \longrightarrow Pic(X)\otimes_\ZZ \OO_{X} \longrightarrow 0 
\end{equation}
and its dual
\begin{equation} \label{TX}0\longrightarrow \OO_X^{\oplus s} \longrightarrow \bigoplus_{i=1}^{n+s}\OO_X(D_i) \longrightarrow \TT X \longrightarrow 0,
\end{equation}
where we are using the isomorphism $Pic(X)\simeq \ZZ^s$ of Remark \ref{fijariso} in order to get $Pic(X)\otimes_\ZZ \OO_X\simeq \OO_X^{\oplus s}$.
The first arrow in the second sequence corresponds to the \textit{radial} vector fields
\[ R_t = \sum_{i=1}^{n+s}a_i^t x_i \frac{\partial}{\partial x_i} \hspace{0.2cm} 1\leq t\leq s.\]
Observe that the coefficients in the radial vector fields depend on the choice of the isomorphism $Pic(X)\simeq \ZZ^s$ made in Remark \ref{fijariso}.
Let $\alpha\in Pic(X)$ and $\omega\in H^0(X,\Omega^1_X(\alpha))$. We will say that $\omega$ is a twisted differential form \textit{of degree} $\alpha$ and denote it $\deg(\omega)=\alpha$. As in the projective case, we can use the Euler sequence in order to get a description of $\omega$ in homogeneous coordinates of the form 
\[ \omega=\sum_{i=1}^{n+s} A_i(x) dx_i, \]
where $A_i\in S_{\alpha-[D_i]}$ and $\ii_{R_t}\left(\sum A_i dx_i\right)=0$ for every $1\leq t\leq s$. In the same spirit, every twisted differential $q$-form of degree $\alpha$ can be described in homogeneous coordinates by an homogeneous differential form of the form 
$$\sum_{|I|=q} A_I dx_I,$$
where $A_I\in S_{\alpha-\sum_{i\in I}[D_i]}$ and whose contraction by every radial vector field $R_t$ is zero.
Analogously, for $f\in S_\alpha$ the degree of the affine vector field $Y=f \frac{\partial }{\partial x_j}$ is $\deg(Y)=\alpha-\deg(x_j)$. With this in mind, one can read the first sequence in the following way: the differential forms in $X$ are homogeneous polynomial  
forms $\omega$ satisfying $\ii_{R_t}\omega=0$ for each radial vector field. 

In the case of twisted vector fields on the other hand if we tensor (2) by $\LL\in  Pic(X)$ and take cohomology we get the exact sequence 
\begin{align*} 0\rightarrow H^0(X,\LL))^{\oplus s} \rightarrow&\bigoplus_{i=1}^{n+s} H^0(X, \LL\otimes\OO_X(D_i)) \xrightarrow{\rho_\LL } H^0(X,\TT X(\LL))\rightarrow    \\
&\rightarrow H^1(X,\LL)^{\oplus s}    \rightarrow \bigoplus_{i=1}^{n+s} H^1(X, \LL\otimes\OO_X(D_i)) \rightarrow \cdots
\end{align*}

\begin{obs}\label{poli}The image of the morphism
\[\rho_\LL:\bigoplus_{i=1}^{n+s}H^0(X, \LL\otimes\OO_X(D_i))\to H^0(X,\TT X(\LL)) \]
above consists exactly of the elements $Y\in H^0(X,\TT X(\LL))$ that can be described in homogeneous coordinates in the form $Y=\sum_{i=1}^{n+s} g_i \frac{\partial}{\partial x_i}$, where $g_i$ is a polynomial of degree $\deg(g_i)=\LL+[D_i]$. This description is unique up to linear combinations of the radial vector fields.
 We will restrict ourselves to these elements. Of course, if $H^1(X,\LL)=0$ then every global section of $\TT X(\LL)$ is of this form. This is not too restrictive, as Demazure and Batyrev-Borisov Vanishing Theorems (Theorem 9.2.3 and Theorem 9.2.7 in \cite{CLS}, respectively) tell us.
\end{obs}

For a similar treatment of these topics (in the case of $1$-dimensional foliations) see \cite{C}.

\begin{ej} $\PP^n=\PP^n(\CC)$. One can think of the classical projective space as the toric variety associated to a complete fan $\Sigma$ in $\RR^n$ with set of rays $\Sigma(1)=\{e_1,\dots,e_n,-e_1-\dots-e_n \}$. If we apply the above construction to this case we get the classical presentation $\PP^n\simeq \CC^{n+1}\setminus\{0\} /\CC^*$. 
\end{ej}

\begin{ej} \label{blowup}Let $X=Bl_p(\PP^n)$ be the blow-up of the usual projective space at a point, say $p=[0:\dots:0:1]$. Toric geometry provides a natural way of blowing up $T$-invariant subvarieties, namely the 
\textit{star subdivision} of the corresponding cone (see Chapter 3, \cite{CLS}). By these means, we get a geometric quotient
\[Bl_p(\PP^n) \simeq \CC^{n+2}\setminus Z /(\CC^*)^2, \]
where $(t_1,t_2)\cdot p= (t_1 p_1,\dots,t_1 p_{n},t_2 p_{n+1}, t_1 t_2 p_{n+2})$ and $Z$ is the union of the linear varieties $V(x_{n+1},x_{n+2})$ and $ V(x_1,\dots,x_n)$. Also, 
the set of $T$-invariant divisors consists of the exceptional divisor $D_{n+1}$ and the closure of the usual hyperplanes in $\PP^n$. Its Picard group is therefore isomorphic to $\ZZ^2$ and the 
grading can be defined by 
$\deg(x_i) =(1,0)$ for every $1\leq i \leq n$, $\deg(x_{n+1})=(0,1)$ and $\deg(x_{n+2})=(1,1)$.

\end{ej}

\begin{ej} The Hirzebruch surface $\HH_r$ is the toric variety defined by the complete fan with rays $\Sigma(1)=\{e_1, e_2,-e_2,-e_1+r e_2 \}$. In this case, the corresponding quotient is
\[ \HH_r\simeq \CC^4\setminus Z / (\CC^*)^2 ,\]
where $Z=V(x_1,x_4)\cup V(x_2,x_3)$ and the action of $(\CC^*)^2$ is defined by $(t_1,t_2)\cdot p = (t_1 p_1, t_2 p_2, t_1^rt_2 p_3, t_1 p_4)$. As for the grading, its Picard group is isomorphic to $\ZZ^2$. Under an appropriate isomorphism, we have
\[ \deg(x_1)=(0,1), \deg(x_2)=(1,0), \deg(x_3)= (1,r) \hspace{0.1cm} \mbox{and} \hspace{0.1cm}\deg(x_4)=(0,1).\]
\end{ej}

\subsection{Preliminaries on differential forms}

In this section we will prove the facts regarding differential forms that will be needed for the following proofs. Along the rest of the article we will use the following notation in order to make calculations more feasible: if $w_1,\dots,w_l$ are vector fields, then $w=w_1\wedge\dots\wedge w_l$ and  $\widehat{w}_i=w_1\wedge\dots\wedge w_{i-1} \wedge w_{i+1}\wedge \dots\wedge w_l$. The symbol $\Omega$ will stand for the $(n+s)$-form $dx_1\wedge\dots\wedge dx_{n+s}$. For a differential form $\omega$, we will denote $S(\omega)\subseteq \CC^{n+s}$ its zero-locus.

We will begin by stating a variant of Euler's formula that will be very useful for our calculations:
\begin{lema} \label{euler} Let $\omega$ be a twisted differential form on $X$ of degree $\deg(\omega)=(c_1,\dots,c_s)\in Pic(X)$. Then for every $1 \leq k \leq s$ we have
\[ \ii_{R_k}d\omega=c_k \omega. \]
\end{lema}
\begin{dem}
Before beginning the proof let us recall some facts about how the interior product relates to the Lie 
derivative. Let $Z$ be a vector field and $\eta$ a differential form. Cartan's 
identity states that 
\begin{equation}\label{Cartan} \LL_Z\eta= d\ii_Z \eta + \ii_Z d\eta.
\end{equation}
Since $\omega$ descends to $X$, applying the above equation for $Z=R_k$ and $\eta=\omega$ we get 
$$\ii_{R_k}d\omega=\LL_{R_k}\omega=\frac{\partial}{\partial t_k}\vert_{t=1} \chi^{\deg(\omega)}(t) \hspace{0.1cm}\omega.$$ 
The equality $c_k=\frac{\partial}{\partial t_k}\vert_{t=1} \chi^{\deg(\omega)}(t)$ follows from $\chi^{\deg(\omega)}(t)=\prod_{i=1}^s t_i^{c_i}$.  
\end{dem}

Recall that every affine differential $q$-form $\omega\in \Omega^q_{\CC^{n+s}}$ defines a map $\omega: \TT\CC^{n+s}\to \Omega^{q-1}_{\CC^{n+s}}$ such that for every local section $Z$ we have $\omega(Z)=\ii_Z\omega$.
The element $\omega$ is said to be \textit{integrable} if  $\ker(\omega)\subseteq \TT\CC^{n+s}$ is closed under the Lie bracket of vector fields.

\begin{lema} Let $\X_1,\dots,\X_{n+s-q}$ be polynomial vector fields in $\CC^{n+s}$. If $\omega=\ii_\X\Omega$ is integrable then there exist rational functions $f_1,\dots,f_{n+s-q}$ satisfying
\[ d\omega = \sum_{j=1}^{n+s-q} f_j \ii_{\widehat{\X_j}}\Omega \]
\end{lema}
\begin{dem}
If $Y$ and $Z$ are vector fields, we can describe the commutator of $\LL_Z$ and 
$\ii_Y$ with the identity
\begin{equation} \label{comm}\left[ \LL_Z, \ii_Y \right]=\ii_{\left[Z,Y\right]}. 
\end{equation}
Now suppose we have polynomial vector fields $\X_1,\dots, \X_{n+s-q}$ satisfying the hypothesis of the lemma. For every ordered subset $J=\{ j_1,\dots,j_r \}\subseteq \{ 1,\dots,n+s-q \}$ of size $r$ we will denote $\X_J=\X_{j_1}\wedge\dots\wedge\X_{j_r}$. We will now prove - by induction in $r$ - that for every such $J$ there exist rational functions $\{ f_I \}_{|I|=r-1}$ such that 
\[ d\ii_{\X_J}\Omega = \sum_{\substack{I\subseteq \{1,\dots,n+s-q\} \\ |I|=r-1 }} f_I \ii_{\X_I}\Omega.\] 
The lemma will follow by setting $J=\{ 1,\dots,n+s-q \}$. For a subset consisting of a single element, the assertion is trivial since every $(n+s)$-form is a multiple of $\Omega.$  
Let us suppose that the statement holds for every $k< r$ and let $J=\{ j_1\dots,j_r\}$ be some subset. Applying (\ref{Cartan}) for $Z=\X_{j_1}$ and $\eta=\ii_{\X_{j_2}}\dots\ii_{\X_{j_r}}\Omega$ we see that $d\ii_{\X_{j_1}}=\LL_{\X_{j_1}}-\ii_{\X_{j_1}}d$ and 
\begin{align*}
d\ii_{\X_J} &= \left( \LL_{\X_{j_1}} - \ii_{\X_{j_1}}d \right)\ii_{\X_{j_2}}\dots\ii_{\X_{j_r}}\Omega \\
&= \left( \LL_{\X_{j_1}}\ii_{\X_{j_2}} - \ii_{\X_{j_1}}d\ii_{\X_2} \right)\ii_{\X_{j_3}}\dots\ii_{\X_{j_r}}\Omega.
\end{align*}
Equation \ref{comm} tells us that $[\LL_{\X_{j_1}},\ii_{\X_{j_2}}]=\ii_{[{\X_{j_1}},{\X_{j_2}}]}$ and therefore
\[ ( \ii_{[{\X_{j_1}},\X_{j_2}]} + \ii_{\X_{j_2}}\LL_{\X_{j_1}} - \ii_{\X_{j_1}}d\ii_{\X_{j_2}}    )  \ii_{\X_{j_3}}\dots\ii_{\X_{j_r}}\Omega.  \]
Being $\omega=\ii_\X\Omega$ integrable, we have that $[\X_{j_1},\X_{j_2}]\in\ker(\omega)$. Without loss of generality, we can assume that $\{ \X_1,\dots,\X_{n+s-q}\}$ can be extended to a basis of the $\CC(x_1,\dots,x_{n+s})$-vector space of rational vector fields (otherwise $\omega=0$ and the assertion on $d\omega$ would be trivial). It follows by duality that there must exist some rational functions
$\beta_1,\dots, \beta_{n+s-q}$ such that $[\X_{j_1},\X_{j_2}]=\sum_{i=1}^{n+s-q} \beta_i \X_i$. This means that 
\begin{align*} d\ii_{\X_J} &= \left(\sum_{i=1}^{n+s-q}\beta_i \ii_{\X_i}+ \ii_{\X_{j_2}}\LL_{\X_{j_1}} - \ii_{\X_{j_1}}d\ii_{\X_{j_2}} \right) \ii_{\X_{j_3}}\dots\ii_{\X_{j_r}}\Omega  \\
&= \left(\sum_{i=1}^{n+s-q}\beta_i \ii_{\X_i}+ \ii_{\X_{j_2}}d\ii_{\X_{j_1}}+ \ii_{\X_{j_2}}\ii_{\X_{j_1}}d - \ii_{\X_{j_1}}d\ii_{\X_{j_2}} \right) \ii_{\X_{j_3}}\dots\ii_{\X_{j_r}}\Omega.
\end{align*}      
Applying the inductive hypothesis to the sets $\{j_1,j_3,\dots,j_r \}$, $\{j_2,\dots,j_r \}$ and $\{j_3,\dots,j_r \}$ we get the desired expression for $d\ii_{\X_J}$ and the lemma follows.
\end{dem}

\begin{lema}\label{diff} Let $\X_1,\dots,\X_{n-q}$ be homogeneous (with the grading of the Cox ring of $X$) vector fields in $\CC^{n+s}$ . 
If $\omega=\ii_\X\ii_R\Omega$ is integrable and satisfies $codim(S(\omega))\geq 2$, then there exist homogeneous vector fields  
$\wt{\X}_1,\dots,\wt{\X}_{n-q}$ satisfying $\deg (\wt{\X}_i ) =\deg(\X_i)$ and
\begin{enumerate}
  \item $\omega=\ii_{\wt{\X}}\ii_R\Omega$.
  \item $d\omega=\sum_{t=1}^s (-1)^{n-q+t-1} c_t \hspace{0.1cm} \ii_{\wt{\X}}\ii_{\widehat{R}_t}\Omega, $
\end{enumerate}
where $(c_1,\dots,c_s)=\deg(\omega)$.
\end{lema}
\begin{dem} The previous lemma guarantees the existence of some rational functions  
$f_1,\dots,f_{n-q}, a_1,\dots,a_s$ such that
\[d\omega= \sum_{j=1}^{n-q} f_j \imath_{\widehat{\X}_j}\ii_R\Omega + \sum_{t=1}^s a_t \imath_{\X}\ii_{\widehat{R}_t}\Omega      .      \]
The polynomial differential form $ \ii_{\X_i}d\omega = (-1)^{i+1} f_i \omega $ is homogeneous of degree $\deg(\omega)+\deg(\X_i)$ and therefore $f_i$ is homogeneous of degree $\deg(\X_i)$. Let us write $f_i=\frac{h_i}{g_i}$ with $h_i$ and $g_i$ polynomials without common factors. Multiplying by $g_i$ the previous equation we get 
\[ g_i \ii_{\X_i}d\omega = (-1)^{i+1} h_i \omega,\]
and therefore $\omega$ must vanish along $\{g_i=0\}$. But then the hypothesis on $S(\omega)$ implies that $f_i$ is in fact a polynomial. On the other hand, by Lemma \ref{euler} we must have 
\[ \ii_{R_t} d\omega= c_t \omega =(-1)^{n-q+t-1} a_t \omega, \]
and therefore $a_t=(-1)^{n-q+t-1} c_t \in \ZZ$.
Now consider for every $b_1,\dots b_s \in \CC$ and $1\leq i \leq n-q$ the vector field 
\[ \wt{\X}_i = \X_i +(-1)^i f_i \sum_{t=1}^s b_t R_t.\]
Recall that the radial vector fields are of the form $R_t = \sum_{i=1}^{n+s}a_i^t x_i \frac{\partial}{\partial x_i}$ and therefore are homogeneous of degree zero. Hence it is clear that these new vector fields are homogeneous of the desired degree and satisfy $\omega=\ii_{\wt{\X}}\ii_R\Omega$. 

Let us now compute $d\omega$ in terms of these new vector fields. For every $1\leq t\leq s$ the element $\wt{\X}\wedge\widehat{R}_t$ equals
\begin{align*}  \X \wedge \widehat{R}_t + \sum_{j=1}^{n-q} (-1)^j f_j \X_1\wedge&\dots\wedge\X_{j-1}\wedge \sum_{k=1}^s b_k R_k \wedge \X_{j+1}\wedge\dots\wedge \X_{n-q}\wedge\widehat{R}_t \\
&= \X\wedge \widehat{R}_t + \sum_{j=1}^{n-q} (-1)^{n-q+t-1}b_t f_j \widehat{\X}_j \wedge R.
\end{align*}
But then 
\begin{align*}\sum_{t=1}^s a_t \ii_{\wt{\X}} \ii_{\widehat{R}_t}\Omega&= \sum_{t=1}^s a_t \ii_\X\ii_{\widehat{R}_t}\Omega + \sum_{t=1}^s a_t \sum_{j=1}^{n-q} (-1)^{n-q+t-1}b_t f_j \ii_{\widehat{\X}_j} \ii_R \Omega   \\
&=\sum_{t=1}^s a_t \ii_\X\ii_{\widehat{R}_t}\Omega + \left(\sum_{t=1}^s a_t (-1)^{n-q+t-1} b_t \right) \sum_{j=1}^{n-q}f_j \ii_{\widehat{\X}_j} \ii_R \Omega \\
&=\sum_{t=1}^s a_t \ii_\X\ii_{\widehat{R}_t}\Omega + \left(\sum_{t=1}^s c_t  b_t \right) \sum_{j=1}^{n-q}f_j \ii_{\widehat{\X}_j} \ii_R \Omega\\
&=d\omega + \left(-1+\sum_{t=1}^s c_t  b_t\right) \sum_{j=1}^{n-q}f_j \ii_{\widehat{\X}_j} \ii_R \Omega.
\end{align*}
Observe that since $\omega$ is non-zero (and $X$ is complete) we must have $\deg(\omega)\neq (0,\dots,0)$ and therefore we can choose the elements $b_t\in \CC$ satisfying the equation $\sum_{t=1}^s c_t b_t=1$ in order to get the desired expression for $d\omega$. 
\end{dem}

\subsection{Distributions and Foliations}
 
A \textit{singular holomorphic distribution of codimension q} in $X$ is a subsheaf $\TT\DD$ of $\TT X$ such that the quotient $\N_\DD:= \TT X / \TT\DD$ is a torsion-free sheaf of rank $q$.
A \textit{singular holomorphic foliation of codimension q} is a distribution $\FF$ closed under Lie bracket, i.e., $[\TT\FF,\TT\FF]\subseteq \TT\FF$. Following \cite{DM}, a singular distribution $\DD$ can be described as the kernel of a locally decomposable twisted differential form $\omega_\DD$, i.e., an element $\omega_D\in  H^0(X,\Omega^q_X(\alpha))$ satisfying $\codim(S(\omega_\DD))\geq 2$ and 
\[ \tag{1} \ii_v(\omega_\DD) \wedge \omega_\DD=0  \hspace{0.5cm} \forall v\in \bigwedge^{q-1}\mathcal{T}\CC^{n+s}. \]
We will call $\alpha$ the $\textit{degree}$ of such distribution. The \textit{singular set} of $\DD$ is defined as $S(\DD)=S(\omega_\DD)$.

In the same spirit, a twisted differential form $\omega\in   H^0(X,\Omega^q_X(\alpha))$ defines a foliation $\FF_\omega$ if and only if it satisfies $(1)$, $\codim(S(\omega))\geq 2$ and
 \[ \tag{2} \ii_v(\omega) \wedge d\omega=0  \hspace{0.5cm} \forall v\in \bigwedge^{q-1}\mathcal{T}\CC^{n+s} .\]

With the above definitions already settled, we can consider the spaces
\begin{align*} \DD_q(X,\alpha) &=\{ [\omega] \hspace{0.1cm} \vert \hspace{0.1cm} \omega \hbox{ satisfies (1) and } \codim(S(\omega))\geq 2 \}& \subseteq \PP H^0(X,\Omega^q_X(\alpha)), \\
\FF_q(X,\alpha) &= \{[\omega] \in \DD_q(X,\alpha) \hspace{0.1cm} \vert \hspace{0.1cm} \omega \hbox{ satisfies (2)} \}&
\end{align*}
whose points parametrize distributions and foliations of degree $\alpha$ respectively.

An element $x\in S(\DD)$ is called a \textit{Kupka point} if $d(\omega_\DD)(x)\neq 0$. We will denote by $K(\DD)$ the set of Kupka points of $\DD$. 

A codimension $q$ distribution $\DD$ is said to be \textit{split} if its tangent sheaf splits as a sum of line bundles, i.e.,
\[ \TT\DD \simeq \bigoplus_{i=1}^{n-q}\OO_X(\alpha_i). \]
Of course, split foliations are integrable split distributions.
Observe that each line bundle defines a morphism $\OO_X(\alpha_i)\to \TT X$ that can be represented by a twisted vector field 
$\X_i\in H^0(X,\TT X(-\alpha_i))$. It follows that every stalk $\TT \DD_x$ is the free $\OO_{X,x}$-module generated by $\{\mathfrak{X}_1,\dots,\X_{n-q}\}$. If these elements can be described by homogeneous (with respect to the grading of the Cox ring of $X$) polynomial vector fields, then $\DD$ is induced by the homogeneous polynomial differential form 
\[ \omega_\DD = \ii_\X\ii_R dx_1\wedge\dots\wedge dx_{n+s} \in H^0(X;\Omega^q_X(\beta)), \]
where $\beta=  \sum \alpha_i - \omega_X$. 
\begin{obs} \label{obssplit} Conversely, if $\X_1,\dots,\X_{n-q}$ are homogeneous polynomial vector fields such that the element $\omega=\ii_\X\ii_R dx_1\wedge\dots\wedge dx_{n+s}$ does not vanish in codimension one, then the associated distribution $\DD_\omega$ of codimension $q$ will satisfy
\[ \TT\DD_\omega\simeq \bigoplus_{i=1}^{n-q}\OO_X(\alpha_i), \]
where $\alpha_i=-\deg(\X_i)$. This is because the hypothesis on $S(\omega)$ assures that the morphism $\bigoplus_{i=1}^{n-q}\OO_X(\alpha_i) \to \TT\DD_\omega$ defined by the $\X_i$'s will be an isomorphism in codimension $1$ (and hence everywhere).
\end{obs}

A good feature of these distributions is that its singular set $S(\DD)$ is a very particular determinantal variety. Let $A(\X)$ be the matrix whose columns are the coefficients of $\X_1,\dots,\X_{n-q},R_1,\dots,R_s$. By duality,
\[ \left( \ii_{\X\wedge R} dx_1\wedge\dots\wedge dx_{n+s}\right) (p)=0 \Leftrightarrow \X\wedge R (p)=0.\]
Then the singular set coincides with the locus where the vector fields are linearly dependent, i.e., the set where the rank of $A(\X)$ drops. Equivalently, 
$$S(\DD)=V\left(\delta_1,\dots,\delta_{{n+s}\choose{q}} \right)$$
where the $\delta_i$'s are the $(n+s-q)\times(n+s-q)$-minors of $A(\X)$. The following proposition follows directly from Hilbert-Schaps Theorem (Theorem 5.1, \cite{A}):

\begin{prop} \label{CM}Let $\DD$ be a codimension $1$ split distribution in a complete toric variety satisfying $\codim(S(\DD))=2$. Then $S(\DD)\subseteq \CC^{n+s}$ is Cohen-Macaulay. 
\end{prop}

With this in mind, we can give a description of the singular set of a generic codimension 1 split foliation in terms of its Kupka set in the following way:

\begin{prop}\label{K} Let $\FF$ be a codimension $1$ split foliation in a complete toric variety satisfying $codim(S(\FF))=2$ and $codim(S(\FF)\setminus K(\FF))\geq 3$. Then $S(\FF)=\ov{K(\FF)}$.
\end{prop}
\begin{proof} By Proposition \ref{CM}, the singular locus of $\FF$ is equidimensional. Also, the Kupka set of an holomorphic foliation on a complex manifold of dimension $n\geq 3 $
is a smooth subvariety of codimension $2$ whenever $\codim (S(\FF))\geq 2$  (see Proposition 1.4.1 in \cite{LN}). In particular, the hypothesis on the codimension 
of $S(\FF)\setminus {K(\FF)}$ implies that $S(\FF)=\ov{K(\FF)}$.
\end{proof}

\section{Stability}

\subsection{Split foliations}
In this section we will prove the stability results regarding split foliations stated in the introduction. 

The insights in the constructions in \cite{CP} will be very useful for our
purposes. In fact, we will use an analogous algebraic parametrization of the set of split distributions and show that its differential at a generic point is surjective. This same strategy
was also used in \cite{Log1}, \cite{Log2} and \cite{Rat}. As mentioned before, we will restrict ourselves to the cases where twisted vector fields can be expressed in homogeneous coordinates. We will keep using the notation $w=w_1\wedge\dots\wedge w_l$, $\widehat{w}_i=w_1\wedge\dots\wedge w_{i-1} \wedge w_{i+1}\wedge \dots\wedge w_l$ for vector fields $w_1,\dots,w_l$ and $\Omega=dx_1\wedge\dots\wedge dx_{n+s}$

Let $1\leq q \leq n$ be an integer and $\alpha_1,\dots,\alpha_{n-q} \in Pic(X)$ such that $h^1(X,\OO_X(-\alpha_i))=0$ for every $1\leq i \leq n-q$. By Remark \ref{poli}, this implies that every element in $H^0(X,\TT X(-\alpha_i))$ can be described as an homogeneous polynomial vector field (uniquely up to linear combinations of the radial vector fields).

Let $V_i$ be the space of homogeneous polynomial vector fields $\X_i$ of degree $\deg(\X_i)=-\alpha_i$. Consider the multilinear morphism
\[ \Phi : \bigoplus_{i=1}^{n-q} V_i \longrightarrow 
H^0(X, \Omega^q_{X}(\sum \alpha_i-\omega_X)) \]
defined by $\left(\X_1,\dots \X_{n-q}\right)  \longmapsto 
\ii_\X\ii_R\Omega$. The differential of $\Phi$ at $\X$ is 
\[ d\Phi(\X)(Z_1,\dots,Z_{n-q})=  \sum_{j=1}^{n-q} (-1)^{j-1}\ii_{Z_j} \ii_{\widehat{\X}_j}\ii_R \Omega. \]
Indeed, if $\varepsilon$ is a parameter such that $\varepsilon^2=0$, $\X=(\X_1,\dots,\X_{n-q})$ and $Z=(Z_1,\dots,Z_{n-q})$ we have
\begin{align*} \Phi(\X+\varepsilon Z)&= \ii_{(\X_1+\varepsilon Z_1)}\dots\ii_{(\X_{n-q}+\varepsilon Z_{n-q})}\ii_R\Omega \\
&= \Phi(\X)+\varepsilon \sum_{j=1}^{n-q} \ii_{\X_1}\dots\ii_{\X_{j-1}}\ii_{\varepsilon Z_j} \ii_{\X_{j+1}}\dots\ii_{\X_{n-q}}\ii_R \Omega  \\
&= \Phi(\X) + \varepsilon \sum_{j=1}^{n-q} (-1)^{j-1}\ii_{Z_j} \ii_{\widehat{\X}_j}\ii_R \Omega .
\end{align*}  
Let $\UU$ be the open set of $H^0(X,\Omega^q_{X}(\sum \alpha_i-\omega_X))$ where $\codim(S(\DD))\geq 2$.
Remark \ref{obssplit} and the preceding discussion establish that the set of split distributions with splitting type $(\alpha_1,\dots,\alpha_{n-q})$ coincides with the image of  $\Phi\rvert_{\Phi^{-1}\UU}$, which contains an open set of $\DD_q(X,\sum \alpha_i-\omega_X)$:

\begin{obs} As mentioned in the proof of Proposition \ref{K}, by [Proposition 1.4.1 in \cite{LN}]  every split codimension $1$ foliation $\FF$ such that $\codim(S(\FF))= 2$ satisfies $\codim(K(\FF))=2$. The next theorem shows that the image of $\Phi$ contains a neighbourhood of $\FF$ if every codimension 2 component of $S(\FF)$ is generically of Kupka type.
\end{obs}

\begin{teo}\label{stab1} Let $X$ be a complete toric variety of dimension $n\geq 3$ and $\alpha_1,\dots\alpha_{n-1}\in Pic(X)$ such that $h^1(X,\OO_X(-\alpha_i))$ for every $1\leq i \leq n-1$. Then for every foliation $\FF\in \FF_1(X, \sum \alpha_i-\omega_X)$ 
satisfying $codim(S(\FF)\setminus \ov{K(\FF)})\geq 3$ and 
\[ \mathcal{T}\FF \simeq \bigoplus_{i=1}^{n-1}\OO_{X}(\alpha_i)\]
there exists a Zariski open set $\UU\subseteq \FF_1(X, \sum \alpha_i-\omega_X)$ containing $\FF$ such that
$\TT\FF'\simeq\TT\FF$ for every $\FF' \in \UU$.
\end{teo}
\begin{dem}  Let $\FF=\left[\omega\right]\in \PP H^0(X, \Omega^1_{X}(\sum \alpha_i-\omega_X))$ be the class of a differential form satisfying the conditions of the theorem.  Let $\X_1,\dots, \X_{n-1}$ be homogeneous polynomial vector fields  such that 
\[ \omega=\ii_\X\ii_R\Omega.   \]
By Lemma \ref{diff} we can assume $d\omega=\sum_{t=1}^s(-1)^{n-q+t-1} c_t \hspace{0.1cm} \ii_{\X}\ii_{\widehat{R}_t}\Omega$. Suppose (without loss of generality) $c_1\neq 0$ and define for $2\leq i \leq s$ 
\[R'_i= R_i+\frac{c_i}{c_1}R_1.\]
These new vector fields satisfy $\sum_t (-1)^{t-1}c_t \widehat{R}_t = c_1 R'_2\wedge\dots\wedge R'_s$ and therefore 
\[ d\omega = \pm c_1 \ii_\X\ii_{R'} \Omega.\]

Now let $\eta$ be an element of the tangent space of $\FF$ in $\FF_1(X, \sum \alpha_i-\omega_X)$, i.e.,  an homogeneous $1$-form of degree $\sum \alpha_i-\omega_X$ such that the first order deformation $\omega_\varepsilon:=\omega+\varepsilon \eta$ satisfies the equation 
 \[  \omega_\varepsilon \wedge d\omega_\varepsilon=0  \hspace{0.5cm} \text{(mod $\varepsilon^2$)} .\]
Since $\omega$ satisfies $\omega\wedge d\omega=0$, this equivalent to 
\[ \eta\wedge d\omega + \omega \wedge d\eta =0 .\]
Differentiating the above equation we can conclude that $d\omega\wedge d\eta=0$. By [Lemma 3.1, \cite{CP}] we get a description of $d\eta$ of the form
\[ d\eta = \sum_{i=1}^{n-1}\ii_{Y_i} \ii_{\widehat{\X}_i}\ii_{R'}\Omega + \sum_{t=2}^s \ii_{Z_t}\ii_\X\ii_{\widehat{R}_t'}\Omega,  \]
where $Y_1,\dots, Y_{n-1}$ and $Z_2,\dots,Z_s$ are polynomial vector fields. Replacing $Y_k$ (resp. $Z_j$) by its homogeneous component of degree $\deg(\X_i)$ (resp. $\deg(R_j)=0$) if necessary, we can suppose that these new vector fields are homogeneous and satisfy $\deg (Y_i)=\deg(\X_i)$ and $\deg(Z_t)=0$. The expression for $d\eta$ will still hold for degree considerations.   Contracting with $R_1$, by Lemma \ref{euler} we get
\[ \pm c_1 \eta = \ii_{R_1}d\eta=\sum_{i=1}^{n-1}\ii_{Y_i}\ii_{\widehat{\X}_i}\ii_R\Omega + \sum_{t=1}^s \ii_{Z_t}\ii_\X \ii_{\widehat{R}_t}\Omega. \]
Since $\eta$ descends to $X$ we have 
\begin{align*} 0&= \ii_{R_j}\eta=\ii_{R_j}\left(\pm c_1^{-1}\sum_{i=1}^{n-1}\ii_{Y_i}\ii_{\widehat{\X}_i}\ii_R\Omega + \sum_{t=1}^s \ii_{Z_t}\ii_\X \ii_{\widehat{R}_t}\Omega \right)  \\
&=\pm c_1^{-1} \ii_{R_j}\ii_{Z_j}\ii_\X\ii_{\widehat{R}_j}\Omega= \pm c_1^{-1}\ii_{Z_j}\omega,
\end{align*}
and therefore $Z_j\in\ker(\omega)$. This means there must exist some (homogeneous) polynomials $f_1^j,\dots,f_{n-1}^j,g_1^j,\dots,g_s^j$ such that $Z_j=\sum_i f_i^j \X_i + \sum_k g_k^j R_k$. But then
\[ \pm c_1 \eta =\sum_{i=1}^{n-1}\ii_{Y_i}\ii_{\widehat{\X}_i}\ii_R\Omega + \left( \sum_{j=1}^s g_j^j\right)\omega. \]
Observe that $g:=\sum_j g_j^j$ has degree zero, so we must have $g\in\CC$.  Since $g\omega=0$ in $\TT_{[\omega]}\FF_1(X,\sum \alpha_i-\omega_X)$, we can conclude that $\eta= \pm c_1^{-1}\sum_{i=1}^{n-1}\ii_{Y_i} \ii_{\widehat{\X}_i}\ii_R\Omega$ is actually in the 
image of the differential of $\Phi$ at  the point $\X=(\X_1,\dots,\X_{n-1})$.

\vspace{0.3cm}

The previous calculations tell us that the differential of our parametrization is generically surjective. As in the end of the proof of [Theorem 1, \cite{CP}], this is sufficient to assure that the image of
$\Phi$ contains a neighbourhood of $\FF$ in $\FF_1(X,\sum \alpha_i-\omega_X)$.
\end{dem}

\begin{obs} In [Section 9, \cite{Quall}], the author makes a quite short proof of the stability of codimension $1$ split foliations in projective spaces. Loosely speaking, the key point of the argument is to observe that the problem of stability becomes much easier after dualizing (i.e., taking annihilators). In order to do so, one has to first assure that a generic split foliation belongs to the open set where the morphism $Inv^X\dashrightarrow iPf^X$ is a rational equivalence (here $Inv^X$ and $iPf^X$ denote the moduli spaces of involutive/integrable subsheaves of $\mathcal{T}X$ and $\Omega_X^1$ respectively. For more details on these objects the reader is referred to [Section 6, \cite{Quall}]). Using Proposition \ref{K} the exact same argument leads to the following generalization: 

\begin{teo} Let $X$ be a smooth complete toric variety and 
\[  0\longrightarrow I(\FF)\longrightarrow \Omega^1_{X\times S|S}\longrightarrow \Omega^1_\FF \longrightarrow 0 \]
be a flat family of codimension 1 integrable Pfaff systems. Suppose further that $0\rightarrow I(\FF)_s \rightarrow\Omega^1_X \rightarrow \Omega^1_\FF \rightarrow 0$ defines a foliation  such that $\TT\FF\simeq \bigoplus_{i=1}^{n-1} \LL_i$ with $h^1(\LL_i \otimes \LL_j^{-1})=0$ for every $i,j$. If $S(\FF)\setminus \overline{K(\FF_s)}$ has codimension greater than 2, then every member of the family defines a split foliation.
\end{teo}
\end{obs}

Now we turn our attention to the case of split distributions of codimension greater than one:

\begin{teo} \label{stab2} Let $q\geq 2$ be an integer, $X$ a complete toric variety of dimension $n\geq 3$ and $\alpha_1,\dots\alpha_{n-q}\in Pic(X)$ such that $h^1(X,\OO_X(-\alpha_i))=0$ for every $1\leq i \leq n-q$. Then for every distribution $\DD\in \DD_q(X, \sum \alpha_i-\omega_X)$ 
satisfying $\codim(S(\DD))\geq 3$ and 
\[ \mathcal{T}\DD \simeq \bigoplus_{i=1}^{n-q}\OO_{X}(\alpha_i)\]
there exists a Zariski open set $\UU\subseteq \DD_q(X, \sum \alpha_i-\omega_X)$ containing $\DD$ such that
$\TT\DD'\simeq\TT\DD$ for every $\DD' \in \UU$.
\end{teo}
\begin{dem} Let $\DD=\left[\omega\right]\in \PP H^0(X, \Omega^q_{X}(\sum \alpha_i-\omega_X))$ be the class of a differential $q$-form satisfying the conditions of the theorem. Let $\X_1,\dots, \X_{n-q}$ be homogeneous polynomial vector fields  such that 
\[ \omega=\ii_\X\ii_R\Omega.   \]
Now let $\eta$ be an element of the tangent space of $\DD$ in $\DD_q(X, \sum \alpha_i-\omega_X)$, i.e.,  an homogeneous $q$-form of degree $\sum \alpha_i-\omega_X$ such that $\omega_\varepsilon:=\omega+\varepsilon \eta$ satisfies the equation 
 \[  \ii_v(\omega_\varepsilon) \wedge \omega_\varepsilon=0  \hspace{0.5cm} \text{(mod $\varepsilon^2$)} \hspace{0.5cm} \forall v\in \bigwedge^{q-1}\mathcal{T}\CC^{n+s} .\]
Since $\omega$ itself satisfies the above equality, this is equivalent to
\[ \ii_v (\eta) \wedge \omega + \ii_v (\omega) \wedge \eta =0 \hspace{0.5cm} \forall v\in \bigwedge^{q-1}\mathcal{T}\CC^{n+s}. \]
We can apply [Lemma 3.2, \cite{CP}] (and replace each vector field by its homogeneous component of the corresponding degree) in order to write $\eta$ in the form
\[ \eta = \sum_{i=1}^{n-q}\ii_{Y_i} \ii_{\widehat{\X}_i}\ii_{R}\Omega + \sum_{t=1}^s \ii_{Z_t}\ii_\X\ii_{\widehat{R}_t}\Omega,  \]
where $Y_1,\dots, Y_{n-q}$ and $Z_2,\dots,Z_s$ are polynomial vector fields such that $\deg (Y_i)=\deg(\X_i)$ and $\deg(Z_t)=0$ . Now if we contract $\eta$ with $R_j$ we get
\[  \ii_{R_j}\eta=\ii_{Z_j}\omega=0. \]
From here we can conclude as in the previous proof.
\end{dem}

\begin{obs}\label{obs} The hypotheses on the vanishing of $h^1(X,\OO_X(-\alpha_i))$ in theorems \ref{stab1} and \ref{stab2} are in order to guarantee that every distribution $\DD$ with tangent sheaf
\[ \TT\DD\simeq \bigoplus_{i=1}^{n-q} \OO_X(\alpha_i) \]
is defined by an element of the form $\omega_\DD=\ii_{\X_1}\dots\ii_{\X_{n-q}} \ii_R\Omega$ for some homogeneous polynomial vector fields $\X_i$ of degree $\deg(\X_i)=-\alpha_i$. The proofs of the theorems above actually imply the following:
If this cohomology groups do not vanish but $\DD$ (resp. $\FF$) satisfies the corresponding hypothesis on its singular set and happens to be defined by a differential form as above, then there exists a neighbourhood $\UU\subseteq \DD_q(X,\sum \alpha_i-\omega_X)$ of $\DD$ (resp. a neighbourhood $\UU\subseteq\FF_q(X,\sum \alpha_i-\omega_X)$ of $\FF$)   such that the same holds for every element in $\UU$.
\end{obs}

\subsection{Equivariant linear pullbacks}
In \cite{CP} the stability of split foliations is used to prove that the pullback of generic degree $d$ foliations by linear morphisms $\PP^{n+m}\dashrightarrow \PP^n$ fill out 
components of $\FF_q(\PP^{n+m},d)$. The aim of this section is to generalize this statement to our setting. First, we shall analyse  the ingredients.

In projective spaces, the $D_i$'s are linearly equivalent and every twisted vector field $Z\in H^0(\PP^n,\mathcal{T}\PP^n(-D_i))$ has -in homogeneous coordinates- constant coefficients. Moreover, the intersection of $k$ of them results in a linearly embedded $\PP^{n-k}$. Although this kind of phenomenon is desirable 
for our purposes, we cannot aspire to encounter such behaviour when dealing with arbitrary (not even smooth) toric varieties. For this reason, we need to emphasize on some special
divisors.

The set of effective line bundles $\eff(X)$ is the $s$-dimensional strictly convex (meaning that it does not contain any non-trivial subspace) closed polyhedral cone generated by the 
classes of the $D_i$'s. Thus we can define the following relation in $Pic(X)$ :
$$\alpha \prec \beta \Leftrightarrow  \alpha-\beta \notin \eff(X).  $$

\begin{deff}\label{defdiv} Let $D_j$ be an invariant $T$-divisor. Then $D_j$ is \textit{maximal} if for every $1\leq i \leq n+s$ we have either $[D_i]\prec [D_j]$ or $[D_i]=[D_j]$. 
\end{deff}
Observe that, since $X$ is complete, for $\alpha,\beta\in Pic(X)$ we have either $\alpha\prec\beta$, $\beta\prec\alpha$ or $\alpha=\beta$.
Of course, the above definition can be expressed in terms of global sections as 
$$\dim H^0(X,\OO_X(D_i-D_j)) =0$$  
for every $D_j\not\sim D_i$. 

\begin{prop} Every toric variety admits a maximal divisor.
\end{prop}
\begin{proof} Let $m$ be the number of $(s-1)$-dimensional faces of $\eff(X)$ and $\phi_1,\dots,\phi_m$ linear operators defining them, i.e., linear morphisms $\phi_i:\RR^s\to \RR$ such that $\phi_i(x)\geq 0$ for every $x\in \eff(X)$ and $\{\phi_i(x)=0\}\cap \eff(X)$ is a $(s-1)$-dimensional face of $\eff(X)$. Now let Y be the space $\RR^m$ equipped with the lexicographic order and consider the function $\phi:\{ 1,\dots,n+s\} \to Y$ defined by 
$$k\mapsto (\phi_1(([D_k]),\dots,\phi_m([D_k])).$$
Clearly, $\phi$ has a maximum at some $k_0$.
Observe that since $\dim_\RR(\eff(X))=s$ the $\phi_i$'s span $\Hom_\RR(\RR^s,\RR)$. In particular, for every $1\leq i \leq n+s$ we have either $[D_i]=[D_{k_0}]$ or $\phi_j(i)<\phi_j(k_0)$ for some $j$. In the latter case, $\phi_j([D_i-D_{k_0}])<0$ and therefore $[D_i]-[D_{k_0}]\notin \eff(X)$. But then $[D_{k_0}]$ is maximal and the proposition follows.
\end{proof}

\begin{obs} If $D_i$ is maximal and linearly equivalent to some $D_j$, then $D_j$ is also maximal.
\end{obs}

We will use the notation $\Delta(i)$ for the set of indices $j$ such that $D_j$ is equivalent to $D_i$. 
Of course, maximal divisors behave nicely with respect to taking arbitrary products of toric varieties:

\begin{prop} Let $X_1, X_2$ be smooth complete toric varieties. If $D$ is maximal in $X_1$ then $D\times X_2$ is maximal in $X_1\times X_2$.
\end{prop}
\begin{proof} Recall that the toric variety $X_1\times X_2$ has $T_1\times T_2$ as open torus (here $T_1$ and $T_2$ stand for the respective tori) with the natural action. 
Clearly, every $T_1\times T_2$-invariant divisor is of the form $D\times X_2$ or $X_1\times D$ for some $T_i$-invariant divisor $D$. 
With this in mind, the proposition follows from Künneth's formula.
\end{proof}

\begin{ej} The projective space $\PP^n$ has one unique (maximal) class of $T$-divisors.
\end{ej}

\begin{ej} \label{blowup2}From Example \ref{blowup} we know that the Cox ring of $Bl_p(\PP^n)$ is $\CC[x_1,\dots,x_{n+2}] $ with grading $\deg(x_i) =(1,0)$, $\deg(x_{n+1})=(0,1)$ and $\deg(x_{n+2})=(1,1)$. In this case, we see that the only maximal $T$-divisor is $D_{n+2}\simeq \PP^{n-1}$, which happens to be the only $T$-invariant hyperplane in $\PP^n$ such that $p\notin D_i$. Moreover, $D_{n+2}$ is numerically effective and by Batyrev-Borisov Vanishing we have $h^1(X,\OO_X(-D_{n+2}))=0$.
\end{ej}

This last example is in fact a special case of a general phenomenon, as the following proposition shows.

\begin{prop} Let $X$ be a smooth complete toric variety of dimension $n\geq 2$. If $D$ is maximal in $X$ and $p\in X\setminus D$ is fixed by the torus action, then $D$ is maximal in $Bl_p(X)$. 
\end{prop}
\begin{proof} Without loss of generality, we can assume that $p$ is the distinguished point corresponding to the cone $\sigma=Cone(e_1,\dots,e_n)$. In this context, the hypothesis
$p\notin D$ is equivalent to $D$ not being the divisor associated to any of the $e_i$'s. Recall that the blow-up $Bl_p(X)\to X$ can be constructed via the star-subdivision of $\sigma$, so 
the only additional $T$-divisor associated to the new rays is the exceptional divisor $E$ associated to the ray generated by $e_1+\dots+e_n$. Since the isomorphism $Bl_p(X)\setminus E\to X\setminus 
\{p\}$ maps each $D_i$ into itself, the restriction of rational functions induces an injection 
$$H^0(Bl_p(X), \OO_{Bl_p(X)}(D_i-D))\to H^0(X, \OO_X(D_i-D)),$$
which is zero by hypothesis. As for the exceptional divisor $E$, the global sections of $\OO_{Bl_p(X)}(E-D)$ restrict to elements in $\Gamma(X,\OO_X(-D))=0$, which is also zero since $X$
is complete.
\end{proof}

\begin{obs} With the notation of the previous proposition, if $D$ is numerically effective in $X$ then the same holds in $Bl_p(X)$.
\end{obs}

\begin{obs}\label{DS} If $\{ D_i\}_{i\in S}$ is a set of $T$-divisors such that $D_S:=\cap_{i\in S} D_{i}$ is not empty and $\tau=Cone(\rho_i \vert i\in S)$, then $D_S$ is the smooth complete toric variety associated to the fan 
$$Star(\tau)=\{ \overline{\sigma} \hspace{0.2cm} \vert \hspace{0.2cm} \tau\leq \sigma .\}$$
in the quotient lattice $N(\tau)=N/\langle N\cap \tau\rangle$. In particular, the $T_{N(\tau)}$- divisors of $D_S$ are exactly the intersections with the other $T$-divisors in $X$. We will use the notation $\pi_S: \CC^{m_S}\setminus Z_S\to D_S$ for the Cox quotient of $D_S$.
\end{obs}

Recall that the preimage of the divisor $D_i$ under the quotient morphism $\pi$ is given by the equation $\{x_i=0\}\subseteq \CC^{n+s}\setminus Z$. For every set $\{ D_i\}_{i\in S}$
consisting of maximal elements we will consider a specific type of  projections $X\dashrightarrow D_S$. Of course, we need to assume that this intersection 
is not empty (or equivalently, that there exists some cone $\sigma\in\Sigma$ containing the corresponding rays). If $\{ T_j\}_{j\notin S}$ are linearly independent operators in $\CC^{n+s}$ 
satisfying $\deg(T_j)=\deg(x_j)$ (this is, $T_j$ depends only of the variables in $\Delta(j)$) we can define the projection $p: \CC^{n+s}\to V(x_k|k\in S)$ such that for every $x\in\CC^{n+s}$ we have
\[  p(x)_j=
        \begin{dcases*}
                T_j(x) & if $j\notin S$\\
                0       & if $j\in S$.
        \end{dcases*}
\]
The hypothesis on the degrees guarantees that this morphism is in fact equivariant. Indeed, for $g\in G$ and $j\notin S$ we have
\[ p( g\cdot x)_j=T_j(g\cdot x) = \chi^{[D_j]}(g) T_j(x) \]
and therefore $p(g\cdot x)=g\cdot p(x)$. This means that $p$ descends to $X$, i.e., we have a commutative diagram of the form 
\[
\xymatrix{
\CC^{n+s}\setminus Z \ar[d]^\pi \ar@{-->}[r]^p & V(\{x_i\}_{i\in S})\setminus Z  \ar[d]^\pi\\
X \ar@{-->}[r] &  D_S. }
\]
In fact,  by Remark \ref{DS} the data concerning the divisors not meeting $D_S$ can be dropped: the morphism $p:X\to D_S$ lifts (in Cox coordinates) to the map 
$$\hat{p}: \CC^{n+s}\setminus Z \dashrightarrow \CC^{m_S}\setminus Z_S$$
induced by the $T_i$'s such that $D_i\cap D_S\neq \emptyset$. The affine set $B(p)\subseteq \CC^{n+s}\setminus Z$ where $p$ is not defined is exactly $p^{-1}(Z_S)$. 
In particular, its codimension is greater or equal than 2 if $p$ is generic. 
\begin{deff} Let $S\subseteq \{ 1,\dots.n\}$ and $p: X\dashrightarrow D_S$ be a dominant morphism. We will say that $p$ is an \textit{equivariant linear projection} if it can be described as above.
\end{deff}

Now let $\omega$ be a twisted differential form in $D_S$ and consider its pullback $p^*\omega$. Observe that after an \textit{equivariant} change of coordinates (i.e., an automorphism of $X$) we can suppose that $p$ is the standard projection, so we can conclude that $\codim (  S(p^*\omega))=\codim(S(\omega))$.

Now we are able to describe foliations whose splitting type involves maximal elements.
Combining Theorem \ref{stab1} and Theorem \ref{stab2} with the definitions above, we are able to point out some specific (not just "split")
irreducible components of the moduli space of foliations. Keep in mind that every foliation in a surface is split.

\begin{corollary}\label{pb}  Let $X$ be a complete toric variety of dimension $n\geq 3$ and $\{D_i\}_{i\in S}$ a set of maximal elements such that
$\dim (D_S) \geq 2$. Let $\beta\in Pic(D_S)$ and $\mathcal{C}\subseteq \FF_q(D_S, \beta )$ an irreducible component. Let $\alpha\in Pic(X)$ be the pullback of $\beta$ by a generic equivariant linear projection. If the generic element of $\mathcal{C}$ satisfies the hypotheses of Theorem \ref{stab1} (for $q=1$) or Theorem \ref{stab2} (for $q>1$)
then there exists an irreducible component of $\FF_q(X,\alpha)$ such that its generic element is a linear pullback of an element of $\C$. 
\end{corollary}
\begin{proof} Without loss of generality, let us assume that $S=\{ 1,\dots,d \}$ and that the invariant divisors meeting $D_S$ are $D_1,\dots,D_{m_S+d}$. For a generic equivariant linear projection $p$ and a generic element $\GG\in \C$ satisfying
\[\TT\GG \simeq \bigoplus_{i=1}^{n-d-q} \OO_{D_S}(\alpha_i), \]
the tangent sheaf of its pullback $\FF=p^*\GG$ is
\[ \TT\FF\simeq \left(\bigoplus_{i=1}^{n-d-q} \OO_X(p^*\alpha_i)\right)\oplus \left(\bigoplus_{i=1}^{d}\OO_X(D_i)\right),  \]
where the new terms correspond to the fibers of $p$. We will now explain this claim: after the \textit{equivariant} change of coordinates $x_j\mapsto T_j(x)$ for $j\notin S$ we can suppose that $p$ is the standard projection.  Let  $\ov{Z}_1,\dots ,\ov{Z}_{n-d-q}$ be twisted vector fields of degree $\deg(\ov{Z}_k)=-\alpha_k\in Pic(D_S)$ inducing the splitting of $\TT\GG$. By hypothesis, these vector fields must admit a description by polynomial homogeneous (with the grading of the Cox ring of $D_S$, with variables $z_{1},\dots,z_{m_S}$) vector fields of the form
\[ \overline{Z}_k=\sum_{j=1}^{m_S}B_j^k(z) \frac{\partial}{\partial z_j} . \]
The description of $p$ in homogeneous coordinates $\CC^{n+s}\setminus Z\dashrightarrow \CC^{m_S}\setminus Z_S$ is the morphism expressed by the formula $\hat{p}(x_1,\dots,x_{n+s})=(x_{d+1},\dots,x_{m_s+d})$. Being the square 
\[
\xymatrix{
 V(\{x_i\}_{i\in S})\setminus Z \ar[d]^\pi \ar@{-->}[r]^p & \CC^{m_S}\setminus Z_S \ar[d]^{\pi_S}\\
D_S \ar@{=}[r] &  D_S }
\]
 commutative, the tangent sheaf of $(\pi\vert_{V(\{x_i\}_{i\in S})})^*\GG$ is generated by the radial vector fields $R_1,\dots,R_s$ and the elements 
\[ Z_k=\sum_{j=d+1}^{m_S+d}B_j^k(x_{d+1},\dots,x_{m_s+d}) \frac{\partial}{\partial x_j} \]
of degree $\deg(Z_k)=-p^*\alpha_k\in Pic(X)$. Since every leaf of $\pi^*\FF$ is a cone with center at a leaf of $(\pi\vert_{V(\{x_i\}_{i\in S})})^*\GG$, for every regular point $x$ the stalk $\TT\FF_x$ must be freely generated by the set $\{ Z_1,\dots, Z_{n-d-q},\frac{\partial}{\partial x_ 1},\dots,\frac{\partial}{\partial x_ d}\}$. Thus these polynomial vector fields induce the claimed splitting of $\TT\FF$.

On the other hand, if $\GG$ is the foliation associated to a differential $q$-form of degree $\beta\in Pic(D_S)$ then $\FF$ is induced by the differential form $p^*\omega$ of degree $p^*\beta\in Pic(X)$. It follows that $\FF$ is the foliation associated to the element 
\[\ii_{\frac{\partial}{\partial x_1}}\dots\ii_{\frac{\partial}{\partial x_d}} \ii_{Z_1},\dots,\ii_{Z_{n-d-q}}\ii_R\Omega \in H^0(X,\Omega^1_X(p^*\beta)). \] 
The discussion above the corollary implies that the codimensions of 
$S(\GG)$ and $S(\GG)\setminus \ov{K(\GG)}$ coincide with the codimensions of $S(\FF)$ and $S(\FF)\setminus \ov{K(\FF)}$ respectively. 
In particular, if $S(\GG)$ satisfies the hypothesis of Theorem \ref{stab1} or Theorem \ref{stab2}, then so does $S(\FF)$.

Now by Remark \ref{obs} if $\FF'$ is sufficiently close to $\FF$ in $\FF_q(X,\alpha)$, then it must be defined by an element  
\[ \omega'=\ii_{\X_1}\dots\ii_{\X_d}\ii_{Y_{d+1}}\ii_{Y_{n-q}}\ii_R\Omega,\]
for polynomial vector fields of degree $\deg(\X_i)=-[D_i]$ and $\deg(Y_k)=-p^*\alpha_k$. The first vector fields must be of the form
\[ \X_i=\sum_{j=1}^{n+s} g_j^i(x) \frac{\partial}{\partial x_j} \]
for homogeneous polynomials $g_j^i$ satisfying $\deg(g_j^i)-[D_j]=-[D_i]$. Our hypothesis on $\Delta(i)$ being maximal simplifies the situation in the following way:
 since $g_j^i\in H^0(X,\OO_X(D_j-D_i))$, we have $g_j^i\in\CC$. Moreover, $g_j^i=0$ for  every $j\notin \Delta(i)$.

Consider the $d \times (n+s)$ matrix $M(\X)$ with rows $g^i_1,\dots,g^i_{n+s}$. There must 
be some subset of $\{1,\dots,n+s\}$ of size $d$ for which the corresponding minor does not vanish (otherwise the tangent sheaf would not have the expected rank). 
This means that after an equivariant change of coordinates we can assume that our vector fields satisfy
$\{\frac{\partial}{\partial x_j}\}_{j\in S}= \{\X_1,\dots,\X_{d}\}$. Without loss of generality, we can also suppose that the vector fields $Y_i$ corresponding to the other terms in 
the splitting type of $\FF$ are orthogonal to the $\X_i$'s. Since this last condition is maintained under Lie bracket, these vector fields define a sub-foliation $\GG''$ of $\FF$ 
whose leaves are parallel to $V(x_i|i\in S)$. Taking $\GG'=\GG''|_{V(x_i|i\in S)}$ we get a foliation in $D_S$ satisfying $S(\FF')=q^{-1}(S(\GG'))$, where $q$ stands
for the standard projection. The foliation $\GG'$ satisfies $\codim(S(\GG'))=\codim(S(\FF'))$ and $\TT\FF'|_{X\setminus S(q)}=\TT q^*\GG'|_{X\setminus S(q)}$ so we must have $\FF'=q^*\GG'$.
\end{proof}

Actually, the proof of Corollary \ref{pb} contains a characterization of the split foliations which can be obtained as pullback by equivariant projections: 

\begin{prop}Let $X$ be a complete toric variety of dimension $n\geq 3$ and $\FF$ a foliation in $X$. Suppose further that 
\[ \mathcal{T}\FF\simeq \left(\bigoplus^{n-|S'|-q}_{j=1} \OO_X(\beta_j)\right)\oplus \left( \bigoplus_{i\in S'}\OO_X(D_i)\right) \]
for some set $\{ D_i\}_{i\in S'}$ consisting of maximal divisors such that for every $i\in S'$ we have  $h^1(X,\OO_X(-D_i))=0$.
Then there exists a set $\{D_j \}_{j\in S}\subseteq\bigcup_{i\in S'}\Delta(i)$ with $|S|\leq|S'|$, a foliation $\mathcal{G}$ in $D_S$ ($\dim D_S\geq 2$),  and an equivariant linear projection $p:X\to D_S $ such that $\FF=p^*\mathcal{G}$.
\end{prop}
\begin{proof} We can repeat the argument in the previous proof, but we may have to pick a smaller $S$ in order to guarantee that $\dim\left(\bigcap_{i\in S} D_i\right)\geq 2$.
\end{proof}

\begin{obs} The set $S$ may not be unique: the same foliation could be a pullback from two non-isomorphic (although birational) $T$-divisors at the same time.
\end{obs}

\begin{ej} As mentioned at the beginning of this section, we can recover the linear pullbacks $\PP^{n+m}\dashrightarrow\PP^n$ a special case of Corollary \ref{pb} by setting $X=\PP^{n+m}$ and $S=\{1,\dots,m\}$. 
\end{ej}

	\begin{ej} Let $\C$ be an irreducible component of $\FF_q(\PP^n, d)$ whose generic element satisfies the hypotheses of Theorem \ref{stab1} (for $q=1$) or Theorem \ref{stab2} (for $q\geq 2$) . Combining Corollary \ref{pb} with Example \ref{blowup2} we can conclude that there exists an irreducible component of $\FF_q(Bl_p(\PP^{n+1}), (d,0))$ such that its generic element is a linear pullback of an element in $\C$.
\end{ej}

\vspace{1cm}

$^1$ \small{Departamento de Matemática-IMAS, FCEyN, Universidad de Buenos Aires, Buenos Aires, Argentina. The author was fully supported by CONICET.\newline
\emph{E-mail address:} svelazquez@dm.uba.ar.}

\end{document}